\documentclass[12pt]{amsart}
\usepackage{tikz-cd}
\usepackage{fancyhdr} 
   
\usepackage{latexsym, amssymb, amscd, amsfonts,pstricks,enumitem,amsmath, young}
\usepackage{latexsym, amssymb, amscd, amsfonts,pstricks,enumitem,amsmath}
\usepackage{amsmath,amsthm,amssymb,mathrsfs}
 
\usepackage[all]{xypic}
\usepackage{ifthen}
\usepackage{longtable}
 
\renewcommand{\subsection}[1]{\vspace{3mm}\refstepcounter{subsection}\noindent{\bf \thesubsection. #1.} }

\newcommand{\np}{\vspace{3mm}\refstepcounter{subsection}\noindent{\bf \thesubsection. }}

\renewcommand{\subsubsection}[1]{\vspace{3mm}\refstepcounter{subsubsection}\noindent{\bf \thesubsubsection. #1.} }

\numberwithin{equation}{section}

\renewcommand{\geq}{\geqslant}
\renewcommand{\leq}{\leqslant}
\newcommand{\Osh}{{\mathcal O}}                        

\renewcommand{\H}{\mathrm{H}}                          

\newcommand{\K}{\mathrm{K}}                            

\newcommand{\kk}{\mathbf{k}}

\newcommand{\Vol}{\operatorname{Vol}}

\newcommand{\KK}{\mathbf{K}}
\newcommand{\FF}{\mathbf{F}}
\newcommand{\PP}{\mathbb{P}} 
\newcommand{\QQ}{\mathbb{Q}} 
\newcommand{\ZZ}{\mathbb{Z}} 

\newtheorem{theorem}{Theorem}[section]

\newtheorem{corollary}[theorem]{Corollary}

\theoremstyle{definition}

\begin{document}

\title{Divisorial instability and Vojta's Main Conjecture for {$\QQ$}-Fano varieties}

\author{Nathan Grieve}
\address{Department of Mathematics, Michigan State University,
East Lansing, MI, USA, 48824}
\email{grievena@msu.edu}%

\begin{abstract} 
We study Diophantine arithmetic properties of birational divisors in conjunction with concepts that surround $\K$-stability for Fano varieties.  There is an interpretation in terms of the barycentres of the Newton-Okounkov bodies.   Our main results show how the property of $\K$-instability, combined with the valuation theoretic characterization thereof, which is made possible by work of K. Fujita and C. Li, implies instances of Vojta's Main Conjecture for Fano varieties.  A main tool in the proof of these results is an arithmetic form of Cartan's Second Main Theorem that has been obtained by M. Ru and P. Vojta.
\end{abstract}

\maketitle

\section{Introduction}\label{Introduction}

\np
Our purpose here, is to indicate how the barycentres of the Newton-Okounkov bodies, together with \emph{divisorial} criterion for \emph{$\K$-instability} for a given $\QQ$-Fano variety, relate to Vojta's Main Conjecture.  The results that we obtain are made possible by insights of K. Fujita \cite{Fujita:2016a}, \cite{Fujita:2019} and C. Li \cite{Li:2017a}.  Recall, that they build on earlier work of Li and Xu \cite{Li:Xu:2014}.

\np
A key theme in this article is that ideas from toric geometry, K-stability and measures of singularities influence results in the direction of Diophantine approximation.  We refer to \cite{Grieve:2018:autissier}, and the references therein, for a more detailed discussion on these and other topics.  

\np
Here, we find it convenient to express our Diophantine approximation results using the language of birational divisors.  The concept of birational divisor appears in work of V. V. Shokurov (for example \cite{Shokurov:1996}).  It is closely related to the theory of divisorial valuations  and the classification problem in general. 
 
\np
That this birational divisor language is important for arithmetic purposes was noted by P.~Vojta \cite{Vojta:1996}.  This birational  viewpoint from \cite{Vojta:1996} is developed further in \cite{Ru:Vojta:2016}.  

\np
In what follows, we state our arithmetic results in the number field setting.  Similar results hold true in the (characteristic zero) function field setting.  Indeed, these viewpoints are well developed in \cite{Grieve:Function:Fields} and \cite{Grieve:2018:autissier}, for example, building on, and applying, earlier results from \cite{Wang:2004} and \cite{Ru:Wang:2012}.

\np  
In particular, we establish Theorem \ref{Vojta:canonical:Fano:not:K:stable} which shows how the property that a Fano variety is not $\K$-stable implies instances of Vojta's Main Conjecture.  At the same time, it shows how these topics are related to the Newton-Okounkov body, of a given Fano variety, with respect to a suitably defined (divisorial) admissible flag.

\np  Let us now turn to the question of Vojta's Main Conjecture within the context of $\K$-stability for Fano varieties.  Theorem \ref{Vojta:canonical:Fano:not:K:stable} below extends and improves upon existing recent related results (including \cite[Corollary 1.2]{Grieve:2018:autissier} and \cite[Corollary 1.13]{Ru:Vojta:2016}).  For a brief summary of the results from $\K$-stability for Fano varieties, we refer to Section \ref{K:stable:Fano}.  

\begin{theorem}\label{Vojta:canonical:Fano:not:K:stable}
Let $\KK$ be a number field and fix a finite set of places $S$ of $\KK$.  Suppose that $X$ is a $\QQ$-Fano variety with canonical singularities, defined over $\KK$, and which is not $\K$-stable.  Then over $X$ there exists a nonzero, irreducible and reduced effective Cartier divisor $E$, which is defined over some finite extension field $\FF / \KK$, with $\KK \subseteq \FF \subseteq \overline{\KK}$, for which the inequalities predicted by Vojta's Main Conjecture hold true in the following sense.  Let $\mathbb{E}$ be the birational divisor that is determined by $E$.  Let 
$$\mathbb{D} = \mathbb{D}_1+\dots+\mathbb{D}_q$$ 
be a birational divisor over $X$ that has the two properties that: 
\begin{enumerate}
\item[(i)]{
the traces of each of the $\mathbb{D}_i$ are linearly equivalent to the trace of $\mathbb{E}$ on some fixed normal proper model $X'$ of $X$, defined over $\FF$; and
}
\item[(ii)]{
the traces of each of these divisors $\mathbb{D}_i$, for $i = 1,\dots,q$, intersect properly on this model $X'$.
}
\end{enumerate}
Let $B$ be a big line bundle on $X$ and let $\epsilon > 0$.  Then the inequality
\begin{equation}\label{Vojta:Inequality:Eqn:intro:b}
\sum_{v \in S} \lambda_{\mathbb{D},v}(x) + h_{\K_X}(x) \leq \epsilon h_B(x) + \mathrm{O}(1)
\end{equation}
is valid for all $\KK$-rational points $x \in X(\KK) \setminus Z(\KK)$ and $Z \subsetneq X$ some proper Zariski closed subset defined over $\KK$.
\end{theorem}

In \eqref{Vojta:Inequality:Eqn:intro:b}, $\lambda_{\mathbb{D},v}(\cdot)$ is the \emph{birational Weil function} of $\mathbb{D}$ with respect to the place $v \in S$.  

\np
Theorem \ref{Vojta:canonical:Fano:not:K:stable} is established in Section \ref{Vojta:Instability:Section}.  It is a consequence of Theorem  \ref{Vojta:canonical:Fano:not:D:stable} together with the valuative criteria for $\K$-stability that has been established by Fujita \cite{Fujita:2016a}, \cite{Fujita:2019} and Li \cite{Li:2017a}. To illustrate the conclusions of Theorems \ref{Vojta:canonical:Fano:not:K:stable} and \ref{Vojta:canonical:Fano:not:D:stable}, we mention one special case of Theorem \ref{Vojta:canonical:Fano:not:D:stable}, which, when applied to projective $n$-space, in particular, recovers the familiar inequalities which are implied by the Second Main Theorem of Schmidt.

\begin{corollary}\label{SecondMain:Vojta:Fano:not:K:stable}
Suppose that a nonsingular $\QQ$-Fano variety $X$ is destabilized by a prime divisor $E \subseteq X$ having field of definition some finite extension $\FF$ of the base number field $\KK$.  Let 
$$D = D_1 + \dots + D_q$$ 
be a divisor on $X$, defined over $\FF$, which has the two properties that:
\begin{enumerate}
\item[(i)]{
each of the $D_i$, for $i = 1,\dots, q$, is linearly equivalent to $E$; and
}
\item[(ii)]{
the divisors $D_1,\dots,D_q$ intersect properly.
}
\end{enumerate}
Fix a finite set of places $S$ of $\KK$, let $B$ be a big line bundle on $X$ and fix $\epsilon > 0$.  Then there exists a proper Zariski closed subset $Z \subsetneq X$ so that the inequality
$$
\sum_{v \in S} \lambda_{D,v}(x) + h_{\K_X}(x) \leq \epsilon h_B(x) + \mathrm{O}(1)
$$
holds true for all $x \in X(\KK) \setminus Z(\KK)$.
\end{corollary}
\begin{proof}
Indeed, by assumption, we have that 
$\beta(-\K_X, E) \geq 1 \text{.}$  
Moreover,  the two conditions (i) and (ii), in Theorem \ref{Vojta:canonical:Fano:not:D:stable}, hold true with 
$X' = X \text{.}$  
Thus the conclusion desired by Corollary \ref{SecondMain:Vojta:Fano:not:K:stable} is evidently implied by that of Theorem \ref{Vojta:canonical:Fano:not:D:stable}.
\end{proof}

\np
As some additional comments which illustrate both the hypothesis and conclusions of Theorem \ref{Vojta:canonical:Fano:not:K:stable} and Corollary \ref{SecondMain:Vojta:Fano:not:K:stable}, suppose that $E \subseteq X$ is a nonzero irreducible reduced effective Cartier divisor on a $\QQ$-Fano variety $X$ and having field of definition the base number field $\KK$.  In this context, note that if 
\begin{equation}\label{divisorial:unstable:inequality}
\beta(-\K_X,E) := \int_0^{\infty} \frac{\Vol_X(-\K_X - t E) }{ \Vol_X(-\K_X) } \mathrm{d} t \geq 1 \text{,}
\end{equation}
then the divisor $-\K_X - E$ is big \cite[Lemma 9.4]{Fujita:2016a}.  In particular, it follows that
$$
\sum_{v \in S} \lambda_{E,v}(\cdot) + h_{\K_X}(\cdot) \leq \mathrm{O}(1)
$$
for all finite sets of places $S$ of $\KK$.  On the other hand, in general, the conclusion of Theorem \ref{Vojta:canonical:Fano:not:K:stable} can be seen by analogy with the \emph{Roth type inequalities} which are given by \cite[Theorem 10.1]{McKinnon-Roth}.  Further, the inequality \eqref{divisorial:unstable:inequality}, for prime divisors $E$ with support on $X$, is very much related to the concept of \emph{divisorial stability} as defined in \cite[Definition 1, p. 542]{Fujita:2016a}.

\np  
Finally, recall that
Theorem \ref{Vojta:canonical:Fano:not:K:stable} is a consequence of the more general Theorem \ref{Vojta:canonical:Fano:not:D:stable} which can be seen as the main result of this article.

\np {\bf Notations and other conventions.} Unless explicitly stated otherwise, we let $\KK$ be a number field and $\FF / \KK$ a finite field extension, with $\KK \subseteq \FF \subseteq {\overline{\KK}}$, for $\overline{\KK}$ a fixed algebraic closure of $\KK$.  By a \emph{variety} over a fixed base field $\kk$, we mean a reduced projective scheme over $\operatorname{Spec}(\kk)$.  By a \emph{model} of a normal variety $X$, we mean a proper birational morphism $\pi \colon X' \rightarrow X$ from a normal variety $X'$.  When no confusion is likely, we omit explicit mention of fields of definition for such models $X'$.  By a \emph{$\QQ$-Fano variety $X$} over $\KK$, we mean that $X$ is a geometrically normal, geometrically irreducible projective variety, defined over $\KK$, having the property that $X_{\overline{\KK}}$ has at most \emph{log-terminal singularities} and, finally, having the property that the \emph{anti-canonical divisor} $-\K_X$ is an ample $\QQ$-Cartier divisor.

\np {\bf Acknowledgements.}  I thank 
Steven Lu, Min Ru, Chi Li, Mike Roth, Daniel Greb, Mattias Jonsson, Julie Wang, Julien Keller, Aaron Levin, Gordon Heier 
and many other 
colleagues for their interest and discussions on related topics.  This work benefited from visits to CIRGET, Montreal, NCTS, Taipei, and the Institute of Mathematics Academia Sinica, Taipei, during the Summers of 2018 and 2019.  It was written while I was a postdoctoral fellow at Michigan State University.  
Finally, I thank an anonymous referee for carefully reading this work and for offering helpful suggestions.

\section{
Preliminaries about $\K$-stability for $\QQ$-Fano varieties
}\label{K:stable:Fano}

\np  In this section, we recall the valuative criterion for $\K$-stability for a given $\QQ$-Fano variety $X$.  Recall, that this criterion is made possible by work of Fujita \cite{Fujita:2016a}, \cite{Fujita:2019}, Li \cite{Li:2017a} and building on earlier work of Li and Xu \cite{Li:Xu:2014}.

\np  Our main interest is when $X$ has canonical singularities and is not $\K$-stable.  Indeed, within this context, Theorem \ref{Fano:K:stability}  implies instances of Vojta's Main Conjecture.  With that in mind, and for later use, here we record Corollary \ref{canonical:Fano:Not:K:Stable} which is deduced immediately from the work of Fujita and Li in the form that we state as Theorem \ref{Fano:K:stability}.

\np  In general, the extent to which such a $\QQ$-Fano variety is $\K$-stable is not obvious.  Fix a prime divisor $E$ over $X$ with field of definition $\FF$ a finite extension of $\KK$ contained in $\overline{\KK}$.  Then $E$  determines a birational divisor $\mathbb{E}$ over $X$ as well as a filtration of the anti-canonical ring.  Denote these respective objects by $\mathbb{E}$ and $\mathcal{F}^\bullet R(X, - \K_X)$.  

\np  
Recall, that the filtration $\mathcal{F}^\bullet R(-\K_X)$ is defined over $\FF$.  It turns out that 
\begin{equation}\label{divisorial:asymptotic}
 \beta(-\K_X,E) := \int_0^\infty \frac{\Vol_{X'}\left( - \pi^* \K_X - t E  \right)}{ \Vol_X \left( -\K_X \right)  } \mathrm{d}t \text{,}
\end{equation}
the \emph{asymptotic volume constant} of $(X, - \K_X)$ with respect to $E$, is related to the property that $X$ is $\K$-stable.  Furthermore, by work of K. Fujita \cite{Fujita:2016a}, \cite{Fujita:2019} and C. Li \cite{Li:2017a} (and others), it is known that $\K$-stability for the polarized variety $(X, - \K_X)$ can be characterized in terms of the \emph{log discrepancies} together with the quantity \eqref{divisorial:asymptotic}.  We state their result as Theorem \ref{Fano:K:stability} below.  Note that in \eqref{divisorial:asymptotic}, we have fixed a model $\pi \colon X' \rightarrow X_{\FF}$ which has the property that $E \subseteq X'$ is a Cartier divisor.

\np  
Recall, that $\mathbb{E}$  is the (prime) birational divisor that is determined by $E$.  Fix a normal proper model 
$\pi \colon X' \rightarrow X_{\FF}$ 
for which $\mathbb{E}$ has trace a Cartier divisor
$\mathbb{E}_{X'} = E \text{.}$
The \emph{discrepancy} of $E$ is the well-defined quantity
\begin{equation}\label{discrepancy:eqn:defn}
a(X,E)  := \operatorname{ord}_{\mathbb{E}_{X'}}(\K_{X'/ X}).
\end{equation}
In \eqref{discrepancy:eqn:defn}, the divisor $\K_{X' / X}$ denotes the \emph{relative canonical divisor}
$$
\K_{X' / X} := \K_{X'} - \pi^*\K_X \text{.}
$$
Recall, that $X$ having \emph{log-terminal singularities} means that all such discrepancies are strictly greater than minus one; the condition that $X$ has \emph{canonical singularities} means that these quantities are all nonnegative.  

\np  
By reformulating results which were established by Fujita, \cite{Fujita:2016a}, \cite{Fujita:2019} and Li, \cite{Li:2017a}, the condition that a $\QQ$-Fano variety $X$ be $\K$-stable can be expressed as Theorem \ref{Fano:K:stability} below.  Before stating this result, using terminology that is consistent with \cite[Definition 1.3]{Fujita:2019}, we say that a Cartier divisor $E$ over $X$, supported on some normal proper model 
$$
\pi \colon Y \rightarrow X_\FF \text{, }
$$
and having field of definition $\FF$, is \emph{dreamy} if the bigraded algebra
$$
\bigoplus_{m,\ell \in \ZZ_{\geq 0}} \H^0\left(Y, - \pi^* rm \K_X - \ell E\right)
$$
is finitely generated for some $r \in \ZZ_{\geq 0}$,  which has the property that $-r \K_X$ is a Cartier divisor.

\np  
Having fixed such terminology, the characterization of $\K$-stability that was established in \cite{Fujita:2019} may be formulated in the following manner.

\begin{theorem}[{\cite[Theorem 1.6]{Fujita:2019}}; see also {\cite[Theorem 3.7]{Li:2017a}}]
\label{Fano:K:stability}
A $\QQ$-Fano variety $(X,-\K_X)$ is not $\K$-stable if and only if the inequality
\begin{equation}\label{Fano:K:unstable:eqn}
a(X,E) \leq \beta(-\K_X, E) - 1
\end{equation}
holds true for at least one such dreamy prime divisor $E$ over $X$, defined over some finite extension of the base number field $\KK$.
\end{theorem}

\begin{proof}
In \cite[Theorem 1.6]{Fujita:2019}, it is shown that such a $\QQ$-Fano variety $X$ is $\K$-stable if and only if 
$$\beta(-\K_X, E) < a(X, E) + 1$$ 
for all such dreamy prime divisors $E$.  (Note that the condition there is expressed in terms of the log discrepancies as opposed to the traditional concept of discrepancies, \cite{Kollar:Mori:1998}, which we adopt here.)  In particular, Theorem \ref{Fano:K:stability}, as stated here, is simply an equivalent formulation of \cite[Theorem 1.6]{Fujita:2019}.
\end{proof}

\np  
Theorem \ref{Fano:K:stability} has the following consequence which we use in the course of proving Theorems  \ref{Vojta:canonical:Fano:not:K:stable} and \ref{Vojta:canonical:Fano:not:D:stable}.

\begin{corollary}\label{canonical:Fano:Not:K:Stable}
Suppose that $X$ is a $\QQ$-Fano variety with canonical singularities.  If $X$ is not $\K$-stable, then
\begin{equation}\label{canonical:Fano:K:unstable:eqn}
\beta(-\K_X, E) \geq 1 
\end{equation}
for at least one prime divisor $E$ over $X$ and having field of definition some finite extension of the base number field $\KK$.
\end{corollary}

\begin{proof}
If $X$ has canonical singularities, then 
\begin{equation}\label{canonical:singularities:hypothesis}
a(X,E) \geq 0
\end{equation} 
and so Corollary \ref{canonical:Fano:Not:K:Stable} follows from Theorem \ref{Fano:K:stability} because combining \eqref{Fano:K:unstable:eqn} with \eqref{canonical:singularities:hypothesis}, we obtain the desired inequality \eqref{canonical:Fano:K:unstable:eqn}.
\end{proof}

\np  
We conclude this section by mentioning three well known examples which illustrate Theorem \ref{Fano:K:stability} and Corollary \ref{canonical:Fano:Not:K:Stable} by direct methods.

\noindent{\bf Examples.} 
\begin{enumerate}
\item[(i)]{
If $X = \PP^n$ and $H$ is a hyperplane, then it follows directly from the definition given in \eqref{divisorial:asymptotic} that
$$\beta(-\K_{X}, H) = 1.$$  
Thus, $X = \PP^n$ is a $\QQ$-Fano variety which is not divisorially stable along $H$.
}
\item[(ii)]{
If $X = \PP^1 \times \PP^1$ and $E$ an exceptional divisor that is obtained by blowing-up a closed point, then 
$$-\K_X = \Osh_X(2,2)$$ 
and so 
$$\beta(-\K_X, E) = 2$$ 
by \cite[Example before Lemma 4.1]{McKinnon-Roth}.  
}
\item[(iii)]{
Let $E_1,\dots, E_q$ be effective Cartier divisors on a $d$-dimensional $\QQ$-Fano variety $X$, defined over $\KK$, and in general position.  Suppose further, that each of the $E_i$ are linearly equivalent to a fixed ample divisor on $X$ and that 
$$-\K_X = E_1 + \dots + E_q.$$  
Then, combining the discussion that follows the Analytic General Theorem of \cite{Ru:Vojta:2016} together with \cite[Proposition 3.1]{Grieve:2018:autissier}, we obtain that 
$$\beta(-\K_X,E_i) = \frac{q}{d + 1}.$$  
}
\end{enumerate}

\section{Arithmetic preliminaries}

\np  
In this article, our main interest in divisorial criterion for $\K$-instability for $\QQ$-Fano varieties is because of its relation to Vojta's Main Conjecture (see Theorems \ref{Vojta:canonical:Fano:not:K:stable} and \ref{Vojta:canonical:Fano:not:D:stable}).  The key points to obtaining these observations, are results that were proved in \cite[Corollary 1.3]{Grieve:2018:autissier} and, independently, in \cite[Corollary 1.13]{Ru:Vojta:2016}.   
We state our result in Section \ref{Vojta:Instability:Section} and, there, we find it convenient to express our result using the birational Weil function formalism.

\np
The main point of this section, is to make precise our arithmetic preliminaries.  At the same time, we develop further the basic theory of birational Weil functions that has been given by Ru and Vojta, in \cite{Ru:Vojta:2016}, and extending earlier insights of Vojta \cite{Vojta:1996}.

\medskip
\noindent
{\bf 
Absolute values, product formula and height functions.
}

\np\label{prod:form:assump}
Let $\KK$ be a number field with algebraic closure $\overline{\KK}$.  We let $M_{\KK}$ be the set of absolute values on $\KK$, which is constructed following the conventions of \cite{Bombieri:Gubler}.  In particular, $M_\KK$ satisfies the product formula with multiplicities equal to one
$$
\prod_{v \in M_\KK} |x|_v = 1 \text{, }
$$
for all $x \in \KK^\times$.  At times, we find it convenient to identify elements of $M_{\KK}$ with the places of $\KK$ that they determine.

\np\label{finite:extension:field:conventions} 
Let $\FF / \KK$ be a finite extension, with $\KK \subseteq \FF \subseteq \overline{\KK}$, and $w$ a place of $\FF$ that lies over $v \in M_\KK$.  Put
$$
||\cdot||_w := |\mathrm{N}_{\FF_w / \KK_v}(\cdot)|_v
$$
for 
$$\operatorname{N}_{\FF_w / \KK_v}(\cdot) \colon \FF_w \rightarrow \KK_v$$ 
the norm from $\FF_w$, the completion of $\FF$ with respect to $w$, to $\KK_v$ the completion of $\KK$ with respect to $v$.  In particular, setting
$$
|\cdot|_{w,\KK} := ||\cdot||_w^{\frac{1}{[\FF_w : \KK_v ]}} \text{,}
$$
we obtain an absolute value on $\FF$ that extends $|\cdot|_v$.

\np
Our conventions about height functions are similar to those of \cite{Grieve:2018:autissier}, \cite{Bombieri:Gubler} and \cite{Lang:Diophantine}.  For example, the logarithmic height of a $\KK$-rational point 
$$\mathbf{x} = [x_0:\dots:x_n] \in \PP^n(\KK)\text{, }$$ of projective $n$-space, with respect to the tautological line bundle $\Osh_{\PP^n}(1)$, is defined to be
\begin{equation}\label{log:height:functions}
h_{\Osh_{\PP^n}(1)}(\mathbf{x}) = \sum_{v \in M_\KK} \log \max_i |x_i|_v.
\end{equation}

\np
By pulling back the quantity \eqref{log:height:functions}, we obtain a concept of (logarithmic) height functions for polarized varieties $(X,L)$ over $\KK$.  In general, by writing an arbitrary line bundle, on $X$ and defined over $\KK$, as the difference of two ample line bundles, we obtain a concept of logarithmic height for each line bundle on $X$, defined over $\KK$. 

\medskip
\noindent
{\bf 
Birational divisors and fields of definition.
}

\np
Let $X$ be a geometrically integral, geometrically normal projective variety over $\KK$.  By a \emph{model} of $X$, we mean a proper birational morphism $Y \rightarrow X$, defined over $\KK$, where $Y$ is a geometrically integral, geometrically normal variety over $\KK$.  More generally, we may also consider models of $X$, or more precisely models of $X_\FF$, that are defined over $\FF$.  We let $\KK(X)$ denote the function field of $X$ and $\FF(X)$ the function field of $X_\FF$.

\np
We represent birational $\QQ$-divisors $\mathbb{D}$ over $X$ as an equivalence class of pairs $(Y,D)$, where $Y$ is a model of $X$ and where $D$ is a $\QQ$-Cartier divisor on $Y$.  In particular, $D$ is the \emph{trace} of $\mathbb{D}$ on $Y$.  Furthermore, the equivalence classes are those for which the equivalence relation is that generated by 
$(Y_1,D_1) \sim (Y_2,D_2)$ 
whenever $Y_1$ dominates $Y_2$ via a morphism 
$\phi \colon Y_1 \rightarrow Y_2$ 
and 
$D_1 = \phi^* D_2$. 
More generally, we may also consider birational divisors $\mathbb{D}$ over $X$, defined over $\FF$.

\medskip
\noindent
{\bf 
Birational Weil functions.
}

\np
We represent the \emph{local birational Weil function} of a birational divisor $\mathbb{D}$, with respect to a place $v$ of $\KK$, by $\lambda_{\mathbb{D},v}(\cdot)$.  Such functions may also be realized as equivalence classes of pairs 
$(Y,\lambda_{D,v}(\cdot))\text{,}$ 
for $(Y,D)$ a pair corresponding to $D$, and $\lambda_{D,v}(\cdot)$ a local Weil function for $D$ with respect to $v$.  

\np  Following the approach of \cite[Section 2.2.1]{Bombieri:Gubler}, the idea is to define local birational Weil functions by using equivalence classes of presentations.  By a \emph{presentation} of a b-Cartier divisor 
$\mathbb{D} = (Y, D)$ 
on $X$, we mean an equivalence class of presentations 
$\mathbf{D} = (Y, \mathcal{D})$.  
Here, 
$\mathcal{D} = (s_D;L;\mathbf{s};M;\mathbf{t})$ 
is a presentation for $D$ considered as a Cartier divisor on $Y$.

\np  The equivalence classes of presentations are those for the equivalence relation that is generated by
$
(Y_1,\mathcal{D}_1) \sim (Y_2,\mathcal{D}_2)
$
whenever $Y_1$ dominates $Y_2$ via a morphism
$
\phi \colon Y_1 \rightarrow Y_2
$
and where $\mathcal{D}_1$ is related to $\mathcal{D}_2$ via
$
\mathcal{D}_1 = \phi^* \mathcal{D}_2 \text{.}
$

\np  Fix a birational equivalence class of presentations denoted as
$[\mathcal{D}] = (Y, \mathcal{D}).$  
We then may form
$
[\lambda_{\mathcal{D}}(\cdot,v)] = (Y,\lambda_{\mathcal{D}}(\cdot,v)),
$
the birational equivalence class of local Weil functions.  Here, $\lambda_{\mathcal{D}}(\cdot,v)$ is the Weil function on $Y$ determined by $\mathcal{D}$.  Again, the equivalence relation is governed by
$
(Y_1,\lambda_{\mathcal{D}_1}(\cdot,v)) \sim (Y_2, \lambda_{\mathcal{D}_2}(\cdot,v))$
whenever
$
\lambda_{\mathcal{D}_1}(\cdot,v) = \phi^* \lambda_{\mathcal{D}_2}(\cdot, v) = \lambda_{\phi^* \mathcal{D}_2}(\cdot,v)
$
for $Y_1$ and $Y_2$ models of $X$ with $Y_1$ dominating $Y_2$ via a morphism 
$ \phi \colon Y_1 \rightarrow Y_2$ 
and, again, the presentations $\mathcal{D}_1$ and $\mathcal{D}_2$ are related by 
$\mathcal{D}_1 = \phi^* \mathcal{D}_2.$

\np  Finally, similar to \cite{Grieve:2018:autissier}, we may construct birational Weil functions $\lambda_{\mathbb{D},v}(\cdot)$ for $\mathbb{D}$ a birational divisor over $X$ and defined over $\FF$.  Indeed, let $w$ be  place of $\FF$ over $v$.  Furthermore, write 
$\mathbf{D} = (Y, \mathcal{D})$ 
for an equivalence class of presentations defined over $\FF$.  Again 
$$\mathcal{D} = (s; M, \mathbf{s}; N, \mathbf{t})$$
is a presentation of $\Osh_Y(D)$, for $Y$ a model of $X$ (defined over $\FF$).  From this perspective, the birational Weil function $\lambda_{\mathbb{D}, v}(\cdot)$ has shape
$$
\lambda_{\mathbb{D}}(\cdot,v) = \max_k \min_\ell \log 
\left| 
\frac{s_k}{t_\ell s} (\cdot) 
\right|_{v,\KK} 
= \max_k \min_\ell \log \left| \left| \frac{s_k}{t_\ell s} (\cdot) \right| \right|_w^{\frac{1}{[\FF_w : \KK_v ] }} \text{.}
$$

\section{Divisorial criterion for $\K$-instability and consequences for Vojta's Main Conjecture}\label{Vojta:Instability:Section}

\np
The purpose of this section, is to discuss Vojta's Main Conjecture within the context of (divisorial) $\K$-instability for Fano varieties.  In doing so, we build on ideas that are contained in a number of previous related results on this topic, including those from \cite{McKinnon-Roth}, \cite{Grieve:2018:autissier} and \cite{Ru:Vojta:2016}.  Throughout this section, we fix a finite set of places $S$ of $\KK$.  Again, $\FF / \KK$ denotes a finite extension of fields, $\KK \subseteq \FF \subseteq \overline{\KK}$.

\np
In more precise terms, let $(X,-\K_X)$ be a $\QQ$-Fano variety defined over $\KK$.  Let $E$ be a divisor over $X$ with field of definition $\FF$.  Fix a normal proper model $\pi \colon X' \rightarrow X_{\FF}$ which has the property that $E \subseteq X'$.  In this context, we observe how the quantity
\begin{equation}\label{beta:quantity}
\beta(-\K_X,E) := \int_0^{\infty} \frac{ \Vol_{X'}(-\pi^*\K_X - t E)}{ \Vol_X(-\K_X) } \mathrm{d}t
\end{equation}
may be used to establish instances of Vojta's Main Conjecture.  Recall, that this quantity \eqref{beta:quantity} admits a description as the first coordinate of the barycenter of the Okounkov body $\Delta(-\pi^*\K_X)$ computed with respect to a suitably defined admissible flag on $X'$ with divisorial component $E$.  We refer to \cite{Fujita:2016a} for a more detailed discussion on that topic.

\np  
Our main result in the direction of Vojta's Main Conjecture reads in the following way.

\begin{theorem}\label{Vojta:canonical:Fano:not:D:stable}
Let $\KK$ be a number field and fix a finite set of places $S$ of $\KK$.  Let $X$ be a $\QQ$-Fano variety defined over $\KK$ and let $E$ be a reduced, irreducible Cartier divisor over $X$ and defined over $\FF$.  Let $\mathbb{E}$ be the birational divisor that is determined by $E$.  Assume that 
$$
\beta(-\K_X,E) \geq 1 \text{.}
$$
Fix a birational divisor of the form
$$
\mathbb{D} = \mathbb{D}_1+\dots+\mathbb{D}_q
$$
which has the following property:  there exists a normal proper model $\pi \colon X' \rightarrow X_{\FF}$ so that if $\mathbb{D}_{i,X'}$, for $i = 1,\dots,q$, and $\mathbb{E}_{X'}$ denote the respective traces of $\mathbb{D}_i$ and $\mathbb{E}$ on $X'$, then:
\begin{enumerate}
\item[(i)]{
the divisors $\mathbb{D}_{1,X'},\dots,\mathbb{D}_{q,X'}, \mathbb{E}_{X'}$ are all linearly equivalent; and
}
\item[(ii)]{
the divisors $\mathbb{D}_{1,X'},\dots,\mathbb{D}_{q,X'}$ intersect properly on $X'$.}
\end{enumerate}
Then, with these assumptions, the inequalities predicted by Vojta's Main Conjecture hold true for $\mathbb{D}$.  Precisely, if $B$ is a big line bundle on $X$ and $\epsilon > 0$, then the inequality
\begin{equation}\label{Vojta:Main:Inequality:eqn:1}
\sum_{v \in S} \lambda_{\mathbb{D},v}(x) + h_{\K_X}(x) \leq \epsilon h_B(x) + \mathrm{O}(1)
\end{equation}
holds true for all $\KK$-rational points $x \in X(\KK) \setminus Z(\KK)$ and $Z \subsetneq X$ some proper Zariski closed subset defined over $\KK$.
\end{theorem}

\begin{proof}
By assumption, and by birational invariance of $\beta(-\K_X,E)$, we have that
\begin{equation}\label{Vojta:Main:Inequality:eqn:2}
\beta(-\K_X,E) = \int_0^{\infty} \frac{ \Vol_{X'}(- \pi^* \K_{X'} - t \mathbb{E}_{X'})}{ \Vol_X(-\K_X) } \mathrm{d}t \geq 1 \text{.}
\end{equation}

We now note that, without loss of generality, by arguing as in \cite[Proof of Theorem 5.1]{Grieve:2018:autissier}, to establish the desired inequality \eqref{Vojta:Main:Inequality:eqn:1}, it suffices to establish, for each given $\epsilon > 0$, the inequality
\begin{equation}\label{Vojta:Inequality:Eqn:reduction}
\sum_{v \in S} \lambda_{\mathbb{D}, v}(x) + h_{\K_X}(x) \leq \epsilon h_{-\K_X}(x) + \mathrm{O}(1)
\end{equation}
for all $\KK$-rational points $x \in X$ outside of some proper Zariski closed subset $Z \subsetneq X$.

To establish \eqref{Vojta:Inequality:Eqn:reduction}, note that it is implied by the inequality
\begin{equation}\label{Vojta:Inequality:Eqn:reduction'}
\sum_{v \in S} \lambda_{\mathbb{D}_{X'},v}(x') \leq (\epsilon + 1)h_{- \pi^* \K_{X}}(x') + \mathrm{O}(1)
\end{equation}
for all $\KK$-rational points $x' \in X'$ outside of some proper Zariski closed subset $Z' \subsetneq X'$.

By linear equivalence of each of the divisors $\mathbb{D}_{1,X'},\dots,\mathbb{D}_{q,X'}, \mathbb{E}_{X'}$, we deduce equality of asymptotic volume constants 
$$
\beta(-\K_X, E) = \beta(-\K_X, \mathbb{E}_{X'}) =  \beta(- \K_X, \mathbb{D}_{i,X'}) \text{, }
$$
for $i = 1, \dots, q$.  In particular, upon applying the Main Arithmetic General Theorem \cite[Theorem 1.1]{Grieve:2018:autissier}, compare also with \cite{Ru:Vojta:2016}, we obtain that the inequality
\begin{equation}\label{Main:Arithmetic:General:Theorem}
\sum_{v \in S} \lambda_{\mathbb{D}_{X'},v}(x') \leq \left( \frac{1}{\beta(-\K_X, E)}  + \epsilon \right) h_{- \pi^* \K_X}(x') + \mathrm{O}(1)
\end{equation}
holds true for all $x' \in X'$ outside of some proper Zariski closed subset $Z' \subsetneq X'$, and defined over $\KK$.

By assumption, we have the inequality \eqref{Vojta:Main:Inequality:eqn:2}; it follows that the righthand side of the inequality \eqref{Main:Arithmetic:General:Theorem} is at most the righthand side of the inequality \eqref{Vojta:Inequality:Eqn:reduction'}.  In particular, by combining these inequalities \eqref{Main:Arithmetic:General:Theorem} and \eqref{Vojta:Inequality:Eqn:reduction'}, we obtain the inequality
\begin{align}\label{Vojta:Inequality:Eqn:reduction''}
\begin{split}
\sum_{v \in S} \lambda_{\mathbb{D}_{X'},v}(x') 
& 
\leq 
\left( \frac{1}{\beta(-\K_X, E)} + \epsilon \right) h_{- \pi^* \K_X}(x') + \mathrm{O}(1) 
\\
&
\leq
(1 + \epsilon) h_{-\pi^* \K_X}(x') + \mathrm{O}(1)
\end{split}
\end{align}
for all $\KK$-rational points $x' \in X'$ outside of some proper Zariski closed subset $Z' \subsetneq X'$.  

The inequality \eqref{Vojta:Inequality:Eqn:reduction''} establishes the inequality \eqref{Vojta:Inequality:Eqn:reduction'} and thus implies the desired inequality which is given by \eqref{Vojta:Main:Inequality:eqn:1}.
\end{proof}

\np  
Finally, we establish our main result in the direction of Vojta's Main Conjecture within the context of $\K$-stability for Fano varieties.  Recall that we stated that result as Theorem \ref{Vojta:canonical:Fano:not:K:stable}.  In particular, this result may also be seen as an improvement to the known previous existing results on this topic, for example \cite[Theorem 10.1]{McKinnon-Roth}, \cite[Corollary 1.3]{Grieve:2018:autissier} and \cite[Corollary 1.13]{Ru:Vojta:2016}.  

\begin{proof}[Proof of Theorem {\ref{Vojta:canonical:Fano:not:K:stable}}]
By assumption, $X$ is a $\QQ$-Fano variety with canonical singularities and is not $\K$-stable.  By Corollary \ref{canonical:Fano:Not:K:Stable}, it follows that
$$
\beta(-\K_X, E) \geq 1
$$
for at least one irreducible and reduced Cartier divisor $E$ over $X$ with field of definition some finite field extension $\FF / \KK$, $\KK \subseteq \FF \subseteq \overline{\KK}$.  In particular, the hypothesis of Theorem  \ref{Vojta:canonical:Fano:not:D:stable} is satisfied and thus the conclusion that is desired by Theorem \ref{Vojta:canonical:Fano:not:K:stable} holds true.
\end{proof}


\begin{thebibliography}{10}

\bibitem{Bombieri:Gubler}
E.~Bombieri and W.~Gubler, \emph{Heights in {D}iophantine geometry}, Cambridge
  University Press, Cambridge, 2006.

\bibitem{Fujita:2016a}
K.~Fujita, \emph{On {$\mathrm{K}$}-stability and the volume functions of
  {$\Bbb{Q}$}-{F}ano varieties}, Proc. Lond. Math. Soc. (3) \textbf{113}
  (2016), no.~5, 541--582.

\bibitem{Fujita:2019}
\bysame, \emph{A valuative criterion for uniform {$\mathrm{K}$}-stability of
  {$\mathbb{Q}$}-{F}ano varieties}, J. reine angew. Math. \textbf{751} (2019),
  309--338.

\bibitem{Grieve:Function:Fields}
N.~Grieve, \emph{Diophantine approximation constants for varieties over
  function fields}, Michigan Math. J. \textbf{67} (2018), no.~2, 371--404.

\bibitem{Grieve:2018:autissier}
\bysame, \emph{On arithmetic general theorems for polarized varieties}, Houston
  J. Math. \textbf{44} (2018), no.~4, 1181--1202.

\bibitem{Kollar:Mori:1998}
J.~Koll\'{a}r and S.~Mori, \emph{Birational geometry of algebraic varieties},
  Cambridge University Press, Cambridge, 1998.

\bibitem{Lang:Diophantine}
S.~Lang, \emph{Fundamentals of {D}iophantine geometry}, Springer-Verlag, New
  York, 1983.

\bibitem{Li:2017a}
C.~Li, \emph{{$\mathrm{K}$}-semistability is equivariant volume minization},
  Duke Math. J. \textbf{166} (2017), no.~16, 3147--3218.

\bibitem{Li:Xu:2014}
C.~Li and C.~Xu, \emph{Special test configuration and {K}-stability of {F}ano
  varieties}, Ann. of Math. (2) \textbf{180} (2014), no.~1, 197--232.

\bibitem{McKinnon-Roth}
D.~McKinnon and M.~Roth, \emph{Seshadri constants, diophantine approximation,
  and {R}oth's theorem for arbitrary varieties}, Invent. Math. \textbf{200}
  (2015), no.~2, 513--583.

\bibitem{Ru:Vojta:2016}
M.~Ru and P.~Vojta, \emph{A birational {N}evanlinna constant and its
  consequences}, Amer. J. Math. (To appear).

\bibitem{Ru:Wang:2012}
M.~Ru and J.~T.-Y. Wang, \emph{An effective {S}chmidt's subspace theorem for
  projective varieties over function fields}, Int. Math. Res. Not. IMRN (2012),
  no.~3, 651--684.

\bibitem{Shokurov:1996}
V.~V. Shokurov, \emph{{$3$}-fold log models}, J. Math. Sci. \textbf{81} (1996),
  no.~3, 2667--2699.

\bibitem{Vojta:1996}
P.~Vojta, \emph{Integral points on subvarieties of semiabelian varieties, {I}},
  Invent. Math. \textbf{126} (1996), 133--181.

\bibitem{Wang:2004}
J.~T.-Y. Wang, \emph{An effective {S}chmidt's subspace theorem for function
  fields}, Math. Z. (2004), no.~246, 811--844.

\end{thebibliography}
\end{document}